\documentclass{article}
\usepackage{fullpage}
\usepackage{amssymb, amsmath, amsthm, enumerate, math tools, tikz-cd}

\DeclareMathOperator{\Z}{\mathbb{Z}}

\theoremstyle{plain}
\newtheorem{thm}{Theorem}
\theoremstyle{definition}
\newtheorem*{defn}{Definition}

\newtheorem{lemma}[thm]{Lemma}
\newtheorem{conj}[thm]{Conjecture}
\begin{document}
\title{On pairs of consecutive sequences with the same radicals}
\author{Noah Lebowitz-Lockard \\
Samueli School of Engineering, \\
University of California, Irvine, \\
5200 Engineering Hall, Irvine, CA, 92697 \\
nlebowit@uci.edu}
\maketitle

\begin{abstract} Let $(m, n, k)$ be a tuple of integers with the property that if $i \leq k$, then $m + i$ and $n + i$ have the same radical. Using a result on the abc Conjecture, we bound $k$ from above, improving a result of Balasubramanian, Shorey, and Waldschmidt. We also bound the number of pairs $(m, n)$ for which $m < n \leq x$ and $m(m + 1) \cdots (m + k - 1))$ and $n(n + 1) \cdots (n + \ell - 1)$ have the same radical and the number of pairs for which $m + i$ and $n + i$ have the same radical for all $i < k$.
\end{abstract}

\section{Introduction}

We define the \emph{radical} of a number as its largest squarefree divisor. In 1963, Erd{\H o}s observed that if $m = 2^k - 2$ and $n = 2^k (2^k - 2)$, then $\textrm{rad} (m) = \textrm{rad} (n)$ and $\textrm{rad} (m + 1) = \textrm{rad} (n + 1)$. He also asked if there were any other examples of this phenomenon \cite[Probl{\` e}me 60]{E1}. Five years later, Makowski \cite{M} found that $(m, n) = (75, 1215)$ is also a solution. As of this writing, no other solutions are known. (See also \cite{E3}, \cite{W} and \cite[\S B19]{G}.)

In light of this fact, numerous authors have considered a weaker statement. Given $m$ and $n$, can we find an upper bound on the largest number $k$ for which $\textrm{rad} (m + i) = \textrm{rad} (n + i)$ for all $i \leq k$? Erd{\H o}s conjectured that there is a fixed constant $K$ such that for all $(m, n)$, there exists a positive integer $i < K$ with $\textrm{rad} (m + i) \neq \textrm{rad} (n + i)$. Woods \cite{Wo} made the stronger conjecture that if $k$ is sufficiently large, then the sequence $\textrm{rad} (n), \textrm{rad} (n + 1), \ldots, \textrm{rad} (n + k - 1)$ uniquely determines $n$. Guy \cite[\S 29]{G} mentions that $k = 3$ might work with at most finitely many exceptions.

Shorey and Tijdeman \cite{ST} obtained a conditional result. The abc Conjecture states that for all $\epsilon > 0$, there are only finitely many pairwise coprime triples $(a, b, c)$ with $a + b = c$ satisfying $\textrm{rad} (abc)^{1 + \epsilon} < c$. Baker \cite{B} provided a heuristic argument for a stronger form of this conjecture.

\begin{conj} Let $\omega = \omega(abc)$. If $a$, $b$, and $c$ satisfy the conditions of the abc Conjecture, then
\[\textrm{rad} (abc) < (6/5) \textrm{rad} (abc) (\log \textrm{rad} (abc))^\omega/(\omega !).\]
\end{conj}

Shorey and Tijdeman answered our problem assuming this result. (Chim, Nair, and Shorey discuss further implications of the previous conjecture in \cite{CNS} and the following result in Section $3.4$ of their paper.)

\begin{thm} If the previous conjecture holds, there are no distinct integers $m$ and $n$ for which $\textrm{rad} (m + i) = \textrm{rad} (n + i)$ for all $i \in \{0, 1, 2\}$.
\end{thm}

Langevin \cite{L} previously proved that simply assuming the abc Conjecture implies that there are at most finitely many exceptions. Unconditionally, much less is known about this problem. Balasubramanian, Shorey, and Waldschmidt \cite{Bal+} obtained the following result.

\begin{thm} \label{main thm} If $(m, n, k)$ is a triple of positive integers with $x < y$ and $\textrm{rad} (m + i) = \textrm{rad} (n + i)$ for all $i \leq k$, then
\[k = \exp(O(\sqrt{\log m \log \log m})).\]
\end{thm}

In this note, we improve this result and obtain the following bound on $k$. Note that while Balasubramanian et al.'s bound depends on $m$, ours depends on $n$.

\begin{thm} \label{main} We have $k \ll (\log n)^{3/2}/(\log \log n)^{9/2}$.
\end{thm}

We also consider a related problem. Erd{\H o}s \cite{E2} asks for the number of pairs of integers $(m, n)$ where $m < n \leq x$ and $\textrm{rad} (m(m + 1)) = \textrm{rad} (n(n + 1))$. We provide an upper bound on this quantity. As far as I am aware, this is the first recorded result on this problem.

\begin{thm} \label{F thm} For a given pair of integers $(k, \ell)$, let $F_{k, \ell} (x)$ be the number of pairs $(m, n)$ with $m < n \leq x$ and
\[\textrm{rad} (m(m + 1) \cdots (m + k - 1)) = \textrm{rad} (n(n + 1) \cdots (n + \ell - 1)).\]
For all $k, \ell > 1$, we have
\[F_{k, \ell} (x) \leq x \exp\left((\ell \log 2 + o(1)) \frac{\log x}{\log \log x}\right)\]
as $x \to \infty$.
\end{thm}

We also prove a variant of this result more closely related to our previous theorems.

\begin{thm} \label{pairs thm} Fix a positive integer $k$. The number of pairs $(m, n)$ with $m < n \leq x$, $\textrm{rad} (m + i) = \textrm{rad} (n + i)$ for all $i \in [0, k - 1]$ is at most
\[x^{1/k} \exp\left((C_k + o(1)) \frac{\log x}{\log \log x}\right)\]
where
\[C_k = \left\{\begin{array}{ll}
2/k, & \textrm{if } k \textrm{ is even,} \\
2/(k - 1), & \textrm{if } k \textrm{ is odd.}
\end{array}\right.\]
\end{thm}

\section{Consecutive strings with equal radicals}

In this section, we bound the largest $k$ for which $\textrm{rad}(m + i) = \textrm{rad}(n + i)$ for all $i \leq k$ for some distinct $m, n \leq x$. Rather than using the abc Conjecture, we use the best current result on the problem.

\begin{thm}[{\cite{StY2}}] There exists a positive constant $C$ such that if $(a, b, c)$ is a pairwise coprime triplet with $a + b = c$, then
\[c < \exp(C \textrm{rad}(abc)^{1/3} (\log \textrm{rad} (abc))^3).\]
\end{thm}

We now prove Theorem \ref{main}.

\begin{proof} Suppose $m$ and $n$ are positive integers with $m < n$. Additionally, suppose that $\textrm{rad} (m + i) = \textrm{rad} (n + i)$ for all $i \leq k$ for some $k$. We bound $k$ from above. We first observe that $m + i \equiv n + i \textrm{ mod rad} (n + i)$ for all $i \leq k$ because $m + i$ and $n + i$ share the same prime factors. Therefore, $m \equiv n \textrm{ mod rad} (n + i)$ for all $i \leq k$. Hence,
\[m \equiv n \textrm{ mod lcm} (\textrm{rad} (n), \textrm{rad} (n + 1), \ldots, \textrm{rad} (n + k)),\]
which in turn implies that
\[m \equiv n \textrm{ mod rad} (n(n + 1) \cdots (n + k)).\]
Because $0 < m < n$, we also have $\textrm{rad} (n(n + 1) \cdots (n + k)) < n$. Additionally,
\[\textrm{rad} (n) \textrm{rad} (n + 1) \cdots \textrm{rad} (n + k) < (2^{k/2} 3^{k/3} 5^{k/5} \cdots P^{k/P}) \textrm{rad} (n(n + 1) \cdots (n + k)),\]
where $P$ is the largest prime $\leq k$. Hence,
\[\textrm{rad} (n) \cdots \textrm{rad} (n + k) \leq e^{k \log k + O(k)} n.\]

There exists some $m < k$ such that $\textrm{rad} (n + m) \textrm{rad} (n + m + 1) \leq (e^{k \log k + O(k)} n)^{2/k} \ll k^2 n^{2/k}$. From here, we can apply the previous theorem. Let $(a, b, c) = (1, n + m, n + m + 1)$. Observe that $c > n$. However, $\textrm{rad}(abc) \ll k^2 n^{2/k}$. Applying our previous theorem gives us
\[n \leq \exp(C_1 k^{2/3} n^{2/(3k)} (\log (k^{2/3} n^{2/k}))^3)\]
for some constant $C_1$. From here we obtain
\[n \ll \exp(C_2 k^{2/3} n^{2/(3k)} ((\log k)^3 + ((\log n)/k)^3))\]
for some $C_2$. If $k \sim C_3 (\log n)^3/(\log \log n)^{9/2}$ for a certain constant $C_3$, then this inequality does not hold. Therefore, if $\textrm{rad} (m + i) = \textrm{rad} (n + i)$ for all $i \leq k$, then $k < C_3 (\log n)^3/(\log \log n)^{9/2}$.
\end{proof}

One notable fact about this proof is that it only works with the most recent bound for the abc Conjecture. (Specifically, an argument of this type only holds if we know that $c < \exp((\textrm{rad} (abc)^{1/2 + o(1)})$.) The previous bound was $c < \exp(\textrm{rad} (abc))^{(2/3) + o(1)})$ \cite{StY1}. Running through the previous argument with this result gives us
\[n \ll \exp(C_4 k^{4/3} n^{2/(3k)}),\]
which holds for all values of $k$.

\section{Pairs with the same radical}

In this section, we prove Theorem \ref{pairs thm}. To do so, we make use of the following result about radicals.

\begin{lemma} \label{lehmer lemma} Fix two integers $Q$ and $x$. The number of $n \leq x$ with $\textrm{rad} (n(n + 1)) | Q$ is at most on the order of $2^{\omega(Q)} \log x$.
\end{lemma}

\begin{proof} By \cite[Thm. 1]{Le}, $\textrm{rad} (n(n + 1))$ can only be a divisor of $Q$ if $n$ has the form $(N - 1)/2$ where $N$ is a solution to one of $2^{\omega(Q)}$ Pell equations. Because the solutions of Pell equations grow exponentially, each equation can only have $O(\log x)$ solutions with $N \leq x$.
\end{proof}

\begin{thm} Fix a positive integer $k$. The number of pairs $(m, n)$ with $m < n \leq x$, $\textrm{rad} (m + i) = \textrm{rad} (n + i)$ for all $i \in [0, k - 1]$ is at most
\[\sqrt[k]{x} \exp\left((C_k + o(1)) \frac{\log x}{\log \log x}\right)\]
with
\[C_k = \left\{\begin{array}{ll}
2/k, & \textrm{ if } k \textrm{ is even,} \\
2/(k - 1), & \textrm{ if } k \textrm{ is odd.}.
\end{array}\right.\]
\end{thm}

\begin{proof} Suppose $m < n \leq x$ with $\textrm{rad} (m + i) = \textrm{rad} (n + i)$ for all $i \in [0, k - 1]$. Additionally, assume that $n > x/2$. We bound the number of such pairs from above. First, we bound the number of possible values of $n$. Then we bound the number of $m$ corresponding to a given $n$.

An argument similar to the one from the start of the previous proof implies that
\[\textrm{rad} (n(n + 1) \cdots (n + k - 1)) \leq x.\]
Therefore, $\textrm{rad} (n + i) < Ckx^{1/k}$ for some positive constant $C$. Let $N(x, y)$ be the number of numbers $\leq x$ with radical $\leq y$. A theorem of Robert and Tenenbaum \cite[p. 208]{Te} states that if $\log y \geq (1 + o(1)) 2^{-3/2} (\log x)^{1/2} (\log \log x)^{3/2}$, then $N(x, y) \sim yF(v)$, where $v = \log(x/y)$ and $\log F(v) \sim 2\sqrt{2 v/\log v}$. Plugging in $y = \sqrt[k]{x}$ gives us $v = (1 - (1/k))(\log x)$ and
\[N(x, Ckx^{1/k}) = x^{1/k} \exp\left((2 + o(1)) \sqrt{2\left(1 - \frac{1}{k}\right)\frac{\log x}{\log \log x}}\right).\]
The number of possible $n \leq x$ with $\min(\textrm{rad} (n), \textrm{rad} (n + 1), \ldots, \textrm{rad} (n + k - 1)) \leq Ckx^{1/k}$ is at most $kN(x, Ckx^{1/k})$.

We now bound the number of possible $m$ corresponding to a given value of $n$. If $k$ is even, then there exists some $i < k - 1$ such that $\textrm{rad} ((n + i)(n + i + 1)) \ll k^2 x^{2/k}$. If $k$ is odd, we have
\[\textrm{rad} ((n + i)(n + i + 1)) \ll k^2 x^{2/(k - 1)}.\]
Without loss of generality, we may assume that $i = 0$. In this case, $\textrm{rad} (m(m + 1)) | \textrm{rad} (n(n + 1))$ because $\textrm{rad} (m(m + 1)) = \textrm{rad} (n(n + 1))$. The previous lemma implies that there are at most $2^{\omega(n(n + 1))} \log x$ possible values of $m$. Because $\textrm{rad} (n(n + 1)) \leq x^{1/k}$, $\omega(n(n + 1)) \ll C_k \log x/\log \log x$. Multiplying $2^{\omega(n(n + 1))} \log x$ by $N(x, Ckx^{1/k})$ gives us our desired result.
\end{proof}

\section{On a question of Erd{\H o}s}

Recall the question that Erd{\H o}s asked from the introduction. How many pairs of numbers $(m, n)$ are there with $m < n \leq x$ and $\textrm{rad} (m(m + 1)) = \textrm{rad} (n(n + 1))$? In this section, we provide an upper bound for this quantity.

\begin{defn} Fix a positive integer $k$. We let $F_{k, \ell} (x)$ be the number of pairs $(m, n) \in \Z^2$ with $m < n \leq x$ satisfying
\[\textrm{rad} (m(m + 1) \cdots (m + k - 1)) = \textrm{rad} (n(n + 1) \cdots (n + \ell - 1)).\]
\end{defn}

We now prove Theorem \ref{F thm}, which we rewrite below. Note that our bound only depends on $\ell$.

\begin{thm} For all $k, \ell > 1$, we have
\[F_{k, \ell} (x) \leq x\exp\left((\ell \log 2 + o(1)) \frac{\log x}{\log \log x}\right).\]
\end{thm}

\begin{proof} Fix $n \leq x$. Suppose $m < n$ with $\textrm{rad} (m(m + 1) \cdots (m + k - 1)) = \textrm{rad} (n(n + 1) \cdots (n + \ell - 1))$. Let $Q = n(n + 1) \cdots (n + k - 1)$. By assumption, $\textrm{rad} (m(m + 1)) | Q$. Lemma \ref{lehmer lemma} implies that there are at most $2^{\omega(Q)} \log x$ values of $m$ satisfying this property. Additionally, every number $\leq x$ has at most $(1 + o(1)) \log x/\log \log x$ distinct prime factors. Therefore, $\omega(Q) \lesssim (k + o(1)) \log x/\log \log x$ as $x \to \infty$. The fact that there are $\lfloor x \rfloor$ choices for $n$ gives us our result.
\end{proof}

Given that this property should be quite rare, one would expect a much smaller upper bound. At present, I do not see a way of even getting $x^{1 - \epsilon}$. One would also expect that $F_{k, \ell} (x)$ \emph{decreases} with $\ell$ because we are placing restrictions on more numbers. We close with a conjecture about the size of $F_{k, \ell} (x)$.

\begin{conj} If $k, \ell > 1$ and $(k, \ell) \neq (2, 2)$, then the equation
\[\textrm{rad} (m(m + 1) \cdots (m + k - 1)) = \textrm{rad} (n(n + 1) \cdots (n + \ell - 1))\] 
only has finitely many solutions.
\end{conj}

\end{document}